\documentclass[11pt,reqno,a4paper]{amsart}
\usepackage[english]{babel}
\usepackage{enumitem}
\setlist[enumerate,1]{label={\textup{(\arabic*)}}}
\setlist[enumerate,2]{label={\textup{(\roman*)}}}
\usepackage{amsmath,amsthm,amssymb}
\usepackage[all]{xy}
\SelectTips{cm}{11}
\usepackage[margin=3cm]{geometry}

\theoremstyle{definition}
\newtheorem{definition}{Definition}[section]
\theoremstyle{plain}
\newtheorem{theorem}[definition]{Theorem}
\newtheorem{proposition}[definition]{Proposition}
\newtheorem{corollary}[definition]{Corollary}
\newtheorem{lemma}[definition]{Lemma}
\theoremstyle{remark}
\newtheorem{remark}[definition]{Remark}
\newtheorem{example}[definition]{Example}

\makeatletter
\newlength{\@thlabel@width}%
\newcommand{\thmenumhspace}{\settowidth{\@thlabel@width}{\upshape(i)}\sbox{\@labels}{\unhbox\@labels\hspace{\dimexpr-\leftmargin+\labelsep+\@thlabel@width-\itemindent}}}
\makeatother

\title{Generalized Jiang and Gottlieb groups}

\author{Marek Golasi\'nski}
\address{Faculty of Mathematics and Computer Science\\
University of Warmia and Mazury\\
S\l oneczna 54 Street\\
10-710 Olsztyn, Poland}
\email{marekg@matman.uwm.edu.pl}

\author[Thiago de Melo]{Thiago de Melo}
\address{Instituto de Geoci\^encias e Ci\^encias Exatas\\ UNESP--Univ Estadual Paulista\\ Av.~24A, 1515, Bela Vista. CEP 13.506--900. Rio Claro--SP, Brazil}
\email{tmelo@rc.unesp.br}

\thanks{Both authors are supported by CAPES--Ci\^encia sem Fronteiras. Processo: 88881.068125/2014-01.}
\subjclass[2010]{Primary:  55Q52, 57M10; secondary:  55Q05, 55Q15}
\keywords{Deck transformation, fibre-preserving map, Gottlieb group, Jiang group, universal covering map, Whitehead center group}

\begin{document}\baselineskip=1.5em

\begin{abstract}Given a map $f\colon X \to Y$, we extend a Gottlieb's result to the generalized Gottlieb group $G^f(Y,f(x_0))$ and show that the canonical isomorphism $\pi_1(Y,f(x_0))\xrightarrow{\approx}\mathcal{D}(Y)$ restricts to
an isomorphism $G^f(Y,f(x_0))\xrightarrow{\approx}\mathcal{D}^{\tilde{f}_0}(Y)$, where $\mathcal{D}^{\tilde{f}_0}(Y)$ is some subset of the group $\mathcal{D}(Y)$ of deck transformations of $Y$ for  a fixed lifting $\tilde{f}_0$ of $f$  with respect to universal coverings of $X$ and $Y$, respectively.
\end{abstract}

\maketitle

\section*{Introduction}

Throughout this paper, all spaces are path-connected with homotopy types of $CW$-complexes. 
We do not distinguish between a map and its homotopy class.

Let $X$ be a connected space, $x_0\in X$ a base-point and $\mathbb{S}^1$ the circle. The  \textit{Gottlieb group} $G(X,x_0)$ of $X$ defined in \cite{gottlieb1}  is the subgroup of the fundamental group $\pi_1(X,x_0)$ consisting of all elements which can be represented by a map $\alpha \colon \mathbb{S}^1\to X$ such that $\mathrm{id}_X \vee \alpha \colon X\vee \mathbb{S}^1\to X$ extends (up to homotopy) to a map $F \colon X\times \mathbb{S}^1\to X$.
Following \cite{gottlieb1}, we recall that $P(X,x_0)$ is the set of elements of $\pi_1(X,x_0)$ whose Whitehead products
with all elements of all homotopy groups $\pi_m(X,x_0)$  are zero for $m\geq 1$. It turns out that $P(X,x_0)$
forms a subgroup of $\pi_1(X,x_0)$ called the \textit{Whitehead center group} and, by \cite[Theorem~I.4]{gottlieb1}, it holds $G(X,x_0)\subseteq P(X,x_0)$.

Now, given a map $f \colon X\to Y$, in view of \cite{gottlieb} (see also \cite{kim}),
the  \textit{generalized Gottlieb group} $G^f(Y,f(x_0))$ is defined 
as the subgroup of  $\pi_1(Y,f(x_0))$ consisting of all elements which can be
represented by a map $\alpha \colon \mathbb{S}^1\to Y$ such that $f\vee \alpha \colon X\vee \mathbb{S}^1\to Y$ extends (up to homotopy) to a map $F \colon X\times \mathbb{S}^1\to Y$. 

The  \textit{generalized Whitehead center group}  $P^f(Y,f(x_0))$ as defined in \cite{kim} consists of all elements $\alpha\in\pi_1(Y,f(x_0))$ whose Whitehead products $[f\beta,\alpha]$ are zero for all  $\beta\in\pi_m(X,x_0)$ with $m\ge 1$. It turns out that $P^f(Y,f(x_0))$ forms a subgroup of $\pi_1(Y,f(x_0))$ and $G^f(Y,f(x_0))\subseteq P^f(Y,f(x_0))\subseteq \mathcal{Z}_{\pi_1(Y,f(x_0))}f_*(\pi_1(X,x_0))$, the centralizer of $f_*(\pi_1(X,x_0))$ in $\pi_1(Y,f(x_0))$. 

If $X=Y$ then the group $G^f(Y,f(x_0))$  is considered in \cite[Chapter~II, 3.5~Definition]{jiang}, denoted by $J(f,x_0)$ and called the \textit{Jiang subgroup} of the map $f\colon Y \to Y$. The role of the $J(f,x_0)$ played in that theory has been intensively studied in the book \cite{brown} as well. More precisely, it is observed that the group $J(f,x_0)$  acts on the right on the set of all fixed point classes of $f$, and any two  equivalent fixed point classes under this action have the same index. Further, Bo-Ju Jiang in \cite[Chapter~II, 3.1~Definition]{jiang} considered also the group $J(\tilde{f}_0)$ for a fixed lifting $\tilde{f}_0$ of $f$ to the universal covering of $Y$ and stressed its importance to the Nielsen--Wecken theory of fixed point classes. 

If $f=\mathrm{id}_X$ then, by \cite[Theorem~II.1]{gottlieb1}, the groups $J(f,x_0)$ and $J(\tilde{f}_0)$ are isomorphic and, according to \cite[Chapter~II, 3.6~Lemma]{jiang}, the groups $J(f,x_0)$ and $J(\tilde{f}_0)$ are isomorphic for any self-map $f\colon X\to X$ but no proof is given.

The aim of this paper is to follow the proof of \cite[Theorem~II.1]{gottlieb1} and give not only a proof of \cite[Chapter~II, 3.6~Lemma]{jiang} but also present a proof of its generalized version for any map $f\colon X \to Y$.

The paper is divided into two sections. Section~\ref{sec.1} follows some results from \cite{gottlieb1} and  deals with some properties of fibre-preserving maps and deck transformations used in the sequel. In particular, we show the functoriality of the fundamental group via deck transformations.  Section~\ref{sec.2} takes up the systematic study of the group $G^f(Y,f(x_0))$. If $X=Y$, $f=\mathrm{id}_X$ and $x_0\in X$ is a base-point then the group $G^f(X,x_0)=G(X,x_0)$ has been described in  \cite[Theorem~II.1]{gottlieb1} via the deck transformation group of $X$ and by \cite[Chapter~II, 3.6~Lemma]{jiang} the groups $G^f(X,x_0)=J(f,x_0)$ and $J(\tilde{f}_0)$ are isomorphic for a self-map $f\colon X\to X$.

Denote by $\mathcal{D}(Y)$ the group of all deck transformations of a space $Y$. Given a map $f\colon X \to Y$, write $\mathcal{L}^f(Y)$ for the set of all liftings  $\tilde{f}\colon \tilde{X}\to \tilde{Y}$ of $f$ with respect to universal coverings of $X$ and $Y$, respectively.  Now, for a fixed $\tilde{f}_0\in \mathcal{L}^f(Y)$, we denote by $\mathcal{D}^{\tilde{f}_0}(Y)$ the set (being a group) of all elements $h\in  \mathcal{D}(Y)$ such that $\tilde{f}_0\simeq_{\tilde{H}} h\tilde{f}_0$, where  $\tilde{H}\colon \tilde{X}\times I \to \tilde{Y}$ is a fibre-preserving homotopy with respect to the universal covering maps $p\colon \tilde{X}\to X$ and $q\colon \tilde{Y}\to Y$. Then, the main result, Theorem~\ref{qq} generalizes \cite[Theorem~II.1]{gottlieb1} and \cite[Chapter~II, 3.6~Lemma]{jiang} as follows:

\textit{Given $f\colon X \to Y$,  the canonical isomorphism $\pi_1(Y,f(x_0))\xrightarrow{\approx}\mathcal{D}(Y)$ restricts to an isomorphism $G^f(Y,f(x_0))\xrightarrow{\approx} \mathcal{D}^{\tilde{f}_0}(Y)$.}

\section{Preliminaries}\label{sec.1}

Let $p\colon X\to A$ and $q\colon Y\to B$ be maps. We say that $f\colon X\to Y$ is a \textit{fibre-preserving map} with respect to $p,q$ provided $p(x)=p(x')$ implies $qf(x)=qf(x')$ for any $x,x'\in X$. 

We say that  $H\colon X\times I \to Y$ is a \textit{fibre-preserving homotopy} with respect to  $p,q$ if $H$ is a fibre-preserving map with respect to $p\times \mathrm{id}_I\colon X\times I  \to A\times I$ and $q\colon Y \to B$.

It is clear that the commutativity of a diagram \[ \xymatrix{ X\ar[r]^{f}  \ar[d]_p & Y\ar[d]^q \\ A \ar[r]^g& B }\] guarantee that $f$ is a fibre-preserving map.

\begin{remark}
(1) Let $p,q,f$ be maps as above. If $p\colon X \to A$ is surjective then there exists a map $g\colon A\to B$ such that $qf=g p$. In addition, if $p$ is a quotient map, then $g$ is continuous.

(2) Given discrete groups $H$ and $K$, consider actions $H\times X \to X$ and $K\times Y \to Y$ and write $p\colon X \to X/H$ and $q\colon Y \to Y/K$ for the quotient maps. If $f\colon X \to Y$ is a $\varphi$-equivariant map for a homomorphism $\varphi\colon H \to K$ then $f$ is a fibre-preserving map with respect to $p$ and $q$.
\end{remark}

If $f\colon X \to Y$ is a fibre-preserving map and $g=\mathrm{id}_A$ then the map $f$ is a fibrewise map in the sense of \cite[Chapter~1]{james}. But, the reciprocal of that does not hold, as it is shown below:

\begin{example}
Let $p=q\colon \mathbb{S}^1\times I \to \mathbb{S}^1$ be the projection. Fix $1\neq \lambda\in \mathbb{S}^1$ and define $f\colon \mathbb{S}^1\times I\to \mathbb{S}^1\times I$ by $f(z,t)= (\lambda z,t)$ for $(z,t)\in \mathbb{S}^1\times I$. Then, $f$ is a fibre-preserving map but clearly $qf\neq p$.
\end{example}

Write $\mathcal{D}(X)$ for the group of all deck transformations of $X$ and  recall that there is an isomorphism $\mathcal{D}(X)\approx \pi_1(X,x_0)$. Next, given  a map $f\colon X\to Y$, consider the set $\mathcal{L}^f(Y)$ of all maps $\tilde{f}\colon \tilde{X}\to \tilde{Y}$ such that the diagram \[ \xymatrix{ \tilde{X}\ar[r]^{\tilde{f}}  \ar[d]_p & \tilde{Y}\ar[d]^q \\ X \ar[r]^f & Y }\] is commutative, where $p,q$ are universal covering maps. 

Fixing $\tilde{f}_0\in \mathcal{L}^f(Y)$, we follow \cite[Chapter~I, 1.2~Proposition]{jiang} to show:

\begin{proposition}\label{star}If $f\colon X\to Y$ then for any lifting $\tilde{f}\colon \tilde{X}\to \tilde{Y}$ of $f$ there is a unique $h\in \mathcal{D}(Y)$ such that $\tilde{f}=h\tilde{f}_0$. 
\end{proposition}
\begin{proof} First, fix $x_0\in X$ and $\tilde{x}_0\in p^{-1}(x_0)$, and
write $y_0=f(x_0)$, $\tilde{y}_0=\tilde{f}_0(\tilde{x}_0)$, and
$\tilde{y}=\tilde{f}(\tilde{x}_0)$. Obviously, $\tilde{y}_0,\tilde{y}\in
q^{-1}(y_0)$. Then, there exists a unique $h\in \mathcal{D}(Y)$ with
$h(\tilde{y}_0)=\tilde{y}$, that is, $h\tilde{f}_0(\tilde{x}_0)=
\tilde{f}(\tilde{x}_0)$. Since both $\tilde{f}$ and $h\tilde{f}_0$ are
lifts of $fp\colon \tilde{X}\to Y$, the unique lifting property
guarantees that $\tilde{f}=h\tilde{f}_0$. 
\par Now, suppose that $\tilde{f}=h\tilde{f}_0=h'\tilde{f}_0$ for some $h,h'\in \mathcal{D}(Y)$. Then, $h\tilde{f}_0(\tilde{x}_0)=h'\tilde{f}_0(\tilde{x}_0)$ implies $h(\tilde{y}_0)=h'(\tilde{y}_0)$. Consequently, $h=h'$ and the proof is complete.
\end{proof}

For a deck transformation $l\in \mathcal{D}(X)$, we notice that $\tilde{f}_0l$ is also a lifting of $f$. By Proposition~\ref{star}, there exists a unique $h_l\in \mathcal{D}(Y)$ such that $\tilde{f}_0l=h_l\tilde{f}_0$. 
Then, we define \[f_*\colon \mathcal{D}(X)\to \mathcal{D}(Y)\] by $f_*(l)=h_l$ for any $l\in \mathcal{D}(X)$. Obviously, the map $f_*$ is a homomorphism. Notice that the map $f_*$ has been already defined in \cite[Chapter~II, 1.1~Definition]{jiang} for any self-map $f\colon X \to X$.

Given $\tilde{f}_1,\tilde{f}_2\in \mathcal{L}^f(Y)$, we define $\tilde{f}_1\ast\tilde{f}_2=h_1h_2\tilde{f}_0$, where $\tilde{f}_1=h_1\tilde{f}_0$ and $\tilde{f}_2=h_2\tilde{f}_0$ for   $h_1,h_2\in \mathcal{D}(Y)$ as in Proposition~\ref{star}. This leads to a group structure on $\mathcal{L}^f(Y)$ with $\tilde{f}_0$ as the identity element. Notice that the groups $\mathcal{L}^f(Y)$ and $\mathcal{D}(Y)$ are isomorphic. In the sequel we identify those two groups, if necessary.

For a homotopy  $\tilde{H}\colon \tilde{X}\times I \to \tilde{Y}$, we write $\tilde{H}_t=\tilde{H}(-,t)$ with $t\in I$.

\begin{lemma}\label{Ht}Let $f\colon X \to Y$. A homotopy $\tilde{H}\colon \tilde{X}\times I \to \tilde{Y}$ with $\tilde{H}_0=\tilde{f}_0$ is a fibre-preserving homotopy if and only if for any $l\in \mathcal{D}(X)$ and $t\in I$ the following diagram 
\[ \xymatrix@C=1.5cm{ \tilde{X}\ar[r]^{\tilde{H}_t}  \ar[d]_-{l} & \tilde{Y}\ar[d]^-{f_*(l)} \\ \tilde{X} \ar[r]^{\tilde{H}_t} & \tilde{Y}  } \] commutes.
\end{lemma}

\begin{proof}Let $(\tilde{x},t),(\tilde{x}',t')\in \tilde{X}\times I$ with  $(p\times \mathrm{id}_I)(\tilde{x},t)=(p\times \mathrm{id}_I)(\tilde{x}',t')$. Then, $p(\tilde{x})=p(\tilde{x}')$ and $t=t'$. Next, consider $l\in \mathcal{D}(X)$ such that $l(\tilde{x})=\tilde{x}'$. Because $\tilde{H}_t=f_*(l)\tilde{H}_tl^{-1}$, we conclude that $q\tilde{H}_t(\tilde{x})=qf_*(l)\tilde{H}_tl^{-1}(\tilde{x})=qf_*(l)\tilde{H}_t(\tilde{x}')=q\tilde{H}_t(\tilde{x}')$. Hence $\tilde{H}$ is fibre-preserving.

Suppose $\tilde{H}$ is fibre-preserving, $l\in \mathcal{D}(X)$ and take $\tilde{x}\in \tilde{X}$, $t\in I$. Then, $\tilde{x}$ and $ l(\tilde{x})$ are in the same fibre of $p$. Since $\tilde{H}_t(\tilde{x})$ and $\tilde{H}_t(l(\tilde{x}))$ are in the same fibre of $q$, there exists a unique $h\in \mathcal{D}(Y)$ such that $h\tilde{H}_t(\tilde{x})=\tilde{H}_tl(\tilde{x})$. If $\varepsilon>0$ is sufficiently small, $h\tilde{H}_{t-\varepsilon}(\tilde{x})=\tilde{H}_{t-\varepsilon}l(\tilde{x})$. Thus, the greatest lower bound of the set of $t$'s such that $h\tilde{H}_t(\tilde{x})=\tilde{H}_tl(\tilde{x})$ must occur when $t=0$. Therefore, by continuity, $h\tilde{H}_0(\tilde{x})=\tilde{H}_0l(\tilde{x})$. But $\tilde{H}_0=\tilde{f}_0$, so we get $h\tilde{f}_0(\tilde{x})=\tilde{f}_0l(\tilde{x})=f_*(l)\tilde{f}_0(\tilde{x})$. This can occur only when $h=f_*(l)$. Consequently, $\tilde{H}_tl=f_*(l)H_t$ and the proof is complete.
\end{proof}

Now, fix $\tilde{f}_0\in \mathcal{L}^{f}(Y)$ and consider the subset $\mathcal{D}^{\tilde{f}_0}(Y)$ of elements $h\in  \mathcal{D}(Y)$  such that $\tilde{f}_0\simeq_{\tilde{H}} h\tilde{f}_0$, where $\tilde{H}\colon \tilde{X}\times I \to \tilde{Y}$ is a fibre-preserving homotopy with respect to the universal covering maps $p,q$. 
Equivalently, in view of Lemma~\ref{Ht}, the set $\mathcal{D}^{\tilde{f}_0}(Y)$ coincides with the set of all elements $h\in  \mathcal{D}(Y)$ for which there is a homotopy $f\simeq_H f$ which lifts to a homotopy $\tilde{H}$ with $\tilde{f}_0\simeq_{\tilde{H}}h\tilde{f}_0$.

Next, write $\mathcal{Z}_{\mathcal{D}(Y)} f_*(\mathcal{D}(X))$ for the centralizer of $ f_*(\mathcal{D}(X))$ in $\mathcal{D}(Y)$. Then, the result below generalizes \cite[Chapter~II, 3.2~Proposition, 3.3~Lemma]{jiang} as follows:

\begin{proposition}\label{prop.center}The subset $\mathcal{D}^{\tilde{f}_0}(Y)$  is contained in $\mathcal{Z}_{\mathcal{D}(Y)} f_*(\mathcal{D}(X))$ and is a subgroup of $\mathcal{D}(Y)$.
\end{proposition}

\begin{proof}
Let $h\in \mathcal{D}^{\tilde{f}_0}(Y)$. Then, there is a fibre-preserving homotopy  $\tilde{H}\colon \tilde{X}\times I \to \tilde{Y}$ with $\tilde{H}_0=\tilde{f}_0$ and  $\tilde{H}_1=h\tilde{f}_0$. But, by Lemma~\ref{Ht}, it holds $f_*(l)\tilde{H}_t=\tilde{H}_tl$ for any $l\in \mathcal{D}(X)$ and $t\in I$. Hence,  for $t=0,1$ we get  $f_*(l)\tilde{f}_0=\tilde{f}_0l$ and $f_*(l)h\tilde{f}_0=h\tilde{f}_0l=hf_*(l)\tilde{f}_0$. Consequently,  $f_*(l)h=hf_*(l)$ and we get that $h\in \mathcal{Z}_{\mathcal{D}(Y)} f_*(\mathcal{D}(X))$.

To show the second part, take $h,h'\in \mathcal{D}^{\tilde{f}_0}(Y)$. Then, there are fibre-preserving homotopies $\tilde{H},\tilde{H}'\colon \tilde{X}\times I \to \tilde{Y}$ with $\tilde{H}_0=\tilde{H}_0'=\tilde{f}_0$, $\tilde{H}_1=h\tilde{f}_0$ and $\tilde{H}'_1=h'\tilde{f}_0$. Next,  consider the map $\tilde{H}''\colon \tilde{X}\times I \to \tilde{Y}$ given by $\tilde{H}''(\tilde{x},t)=hh'^{-1}\tilde{H}'(\tilde{x},1-t)$ for $(\tilde{x},t)\in \tilde{X}\times I$ and notice that $\tilde{H}''$ is a fibre-preserving homotopy with $\tilde{H}''_0=h\tilde{f}_0$ and  $\tilde{H}''_1=hh'^{-1}\tilde{f}_0$. Finally, the concatenation $\tilde{H}\bullet \tilde{H}'' \colon \tilde{X}\times I \to \tilde{Y}$ is a fibre-preserving homotopy with $(\tilde{H}\bullet \tilde{H}'')_0=\tilde{f}_0$ and  $(\tilde{H}\bullet \tilde{H}'')_1=hh'^{-1}\tilde{f}_0$. Consequently, $hh'^{-1}\in \mathcal{D}^{\tilde{f}_0}(Y)$ and the proof is complete.
\end{proof}

Notice that if $f\colon X \to X$ is a self-map then the group $\mathcal{D}^{\tilde{f}_0}(X)$ coincides with the group $J(\tilde{f}_0)$ defined in \cite[Chapter~II, 3.1~Definition]{jiang}.

\section{Main result}\label{sec.2}

Given spaces $X$ and $Y$, write $Y^X$ for the space of continuous maps from $X$ into $Y$ with the compact-open topology. Next, consider the evaluation map $\mathrm{ev} \colon Y^X\to Y$, i.e., $\mathrm{ev}(f)=f(x_0)$ for $f\in  Y^X$ and the base-point $x_0\in X$. 
Then, it holds \[G^f(Y,f(x_0))=\operatorname{Im}\bigl(\mathrm{ev}_\ast \colon \pi_1(Y^X,f)\to \pi_1(Y,f(x_0))\bigr).\] 
Certainly, $G^f(X,f(x_0))$ coincides with the group $J(f,x_0)$ defined in \cite[Chapter~II, 3.5~Definition]{jiang} for a self-map $f\colon X\to X$.
\par Now, we follow \textit{mutatis mutandis} the result \cite[Theorem~II.1]{gottlieb1} to generalize \cite[Chapter~II, 3.6~Lemma]{jiang} as follows: 

\begin{theorem}\label{qq}Given $f\colon X \to Y$, the canonical isomorphism $\pi_1(Y,f(x_0))\xrightarrow{\approx} \mathcal{D}(Y)$ restricts to an isomorphism $G^f(Y,f(x_0))\xrightarrow{\approx} \mathcal{D}^{\tilde{f}_0}(Y)$.
\end{theorem}

\begin{proof}Let $\alpha\in G^f(Y,f(x_0))$ and $h\in \mathcal{D}(Y)$ be the corresponding deck transformation. Then, there is a homotopy $H\colon X\times I \to Y$ such that $H_0=H_1=f$ and $H(x_0,-)=\alpha$, where $x_0\in X$ is a base-point. Next, consider the commutative diagram
\[ \xymatrix@C=1.5cm{
\tilde{X} \ar[d]_{i_0}\ar[rr]^{\tilde{f}_0} && \tilde{Y} \ar[d]^q \\
\tilde{X}\times I \ar[r]^{p\times \mathrm{id}_I} & X\times I \ar[r]^H & Y\rlap{.}
}  \] Then, by the lifting homotopy property there is a map $\tilde{H}\colon \tilde{X}\times I \to \tilde{Y}$ such that $\tilde{H}i_0=\tilde{H}_0=\tilde{f}_0$ and $q\tilde{H}=H(p\times \mathrm{id}_I)$. This implies that $\tilde{H}$ is a fibre-preserving homotopy. Further, because $H_0=H_1=f$, we also derive that $\tilde{H}_0,\tilde{H}_1\in \mathcal{L}^f(Y)$.

Now, since the path $\tilde{\tau}\colon I\to \tilde{Y}$ defined by $\tilde{\tau}=\tilde{H}(\tilde{x}_0,-)$ runs from $\tilde{f}_0(\tilde{x}_0)$ to $\tilde{H}_1(\tilde{x}_0)$, we derive that $\alpha=q\tilde{\tau}$. Consequently, by means of Proposition~\ref{star} we get $\tilde{H}_1=h\tilde{f}_0$ and so $\tilde{H}$ is the required fibre-preserving homotopy with $\tilde{f}_0\simeq_{\tilde{H}} h\tilde{f}_0$. 

Conversely, given $h\in \mathcal{D}^{\tilde{f}_0}(Y)$, there is a fibre-preserving homotopy $\tilde{H}\colon \tilde{X}\times I \to \tilde{Y}$ with $\tilde{f}_0\simeq_{\tilde{H}} h\tilde{f}_0$. This implies a homotopy $H\colon X\times I \to Y$ such that $H_0=H_1=f$ and $q\tilde{H}=H(p\times \mathrm{id}_I)$. Then, the path $\tau\colon I\to Y$ given by $\tau=H(x_0,-)$  leads to the required loop in $G^f(Y,f(x_0))$.
\end{proof}
Notice that by Theorem~\ref{qq} the group $\mathcal{D}^{\tilde{f}_0}(Y)$ is independent of the lifting $\tilde{f}_0\in\mathcal{L}^f(Y)$. Further, the advantage of $G^f(Y,f(x_0))$ over $\mathcal{D}^{\tilde{f}_0}(Y)$ is that it does not involve the covering spaces $\tilde{X}$ and $\tilde{Y}$ explicitly, hence it is easier to handle.

Next, Lemma~\ref{Ht} and Theorem~\ref{qq} yield:
\begin{corollary}\label{cor2}If $f\colon X \to Y$ then $G^f(Y,f(x_0))$ is isomorphic to the subgroup of $\mathcal{D}(Y)$ given by those deck transformations $h$ for which there are homotopies $\tilde{H}\colon \tilde{X}\times I \to \tilde{Y}$ such that $\tilde{f}_0\simeq_{\tilde{H}} h\tilde{f}_0$  and the diagrams 
\[ \xymatrix@C=1.5cm{ \tilde{X}\ar[r]^{\tilde{H}_t}  \ar[d]_-{l} & \tilde{Y}\ar[d]^-{f_*(l)} \\ \tilde{X} \ar[r]^{\tilde{H}_t} & \tilde{Y}  } \] commute for any $l\in \mathcal{D}(X)$ and $t\in I$. Equivalently, the homotopies $\tilde{H}\colon \tilde{X}\times I \to \tilde{Y}$ are $f_*$-equivariant.
\end{corollary}

Let $\mathcal{H}^{\tilde{f}_0}(Y)$ be the subset of all $h\in\mathcal{Z}_{\mathcal{D}(Y)}f_*(\mathcal{D}(X))$ such that $\tilde{f}_0\simeq h\tilde{f}_0$. By similar arguments as in the proof of Proposition~\ref{prop.center}, it is easy to verify that  $\mathcal{H}^{\tilde{f}_0}(Y)$ is a subgroup of $\mathcal{Z}_{\mathcal{D}(Y)}f_*(\mathcal{D}(X))$. 

Now, we process as in the proof of \cite[Theorem~II.6]{gottlieb1} to show:

\begin{proposition}\label{pp}Given $f\colon X \to Y$, there are inclusions \[ G^f(Y,f(x_0))\subseteq \mathcal{H}^{\tilde{f}_0}(Y)\subseteq P^f(Y,f(x_0)). \]
\end{proposition}

\begin{proof} Certainly, the inclusion  $G^f(Y,f(x_0))\subseteq \mathcal{H}^{\tilde{f}_0}(Y)$ is a direct consequence of Proposition~\ref{prop.center} and Theorem~\ref{qq}.

Now, let $h\in \mathcal{H}^{\tilde{f}_0}(Y)$  and $\tilde{H} \colon \tilde{X}\times I\to \tilde{Y}$ be a  homotopy with   $\tilde{f}_0\simeq_{\tilde{H}} h\tilde{f}_0$. Next, consider the path $\tilde{\phi} \colon I\to \tilde{Y}$ defined by $\tilde{\phi}=\tilde{H}(\tilde{x}_0,-)$, where $p(\tilde{x}_0)=x_0$. Then, the loop $\phi=q\tilde{\phi}$ corresponds to $h$.

Notice that $\phi$ acts trivially on $f_\ast(\pi_m(X,x_0))$ for $m>1$ if and only if there is a map $F \colon \mathbb{S}^m\times \mathbb{S}^1\to Y$ such that the diagram
\[ \xymatrix@C=2cm{\mathbb{S}^m\vee \mathbb{S}^1\ar@{^{(}->}[d] \ar[r]^-{f\alpha\vee \phi} & Y \\ 
\mathbb{S}^m\times \mathbb{S}^1 \ar@{-->}[ru]_F &} \]
commutes (up to homotopy) for any $\alpha\in \pi_m(X,x_0)$.\

Given $\alpha\in \pi_m(X,x_0)$ with $m>1$, there exists $\tilde{\alpha}\in\pi_m(\tilde{X},\tilde{x}_0)$ such that $p\tilde{\alpha}=\alpha$. Thus, we define a map \[F' \colon \mathbb{S}^m\times I\xrightarrow{\tilde{\alpha}\times \mathrm{id}_I}\tilde{X}\times I\xrightarrow{\tilde{H}}\tilde{Y}\xrightarrow{q}Y.\]
Because $F'(s,0)=q\tilde{H}(\tilde{\alpha}(s),0)=q\tilde{f}_0(\tilde{\alpha}(s))=fp\tilde{\alpha}(s)$ and
$F'(s,1)=q\tilde{H}(\tilde{\alpha}(s),1)=qh\tilde{f}_0(\tilde{\alpha}(s))=q\tilde{f}_0(\tilde{\alpha}(s))=fp\tilde{\alpha}(s)$ for $s\in \mathbb{S}^m$, the map $F'$
implies the required map $F\colon  \mathbb{S}^m\times \mathbb{S}^1\to Y$. 

Since  $\mathcal{H}^{\tilde{f}_0}(Y)\subseteq \mathcal{Z}_{\mathcal{D}(Y)}f_*(\mathcal{D}(X))$, we derive that $\phi$ acts trivially also on $f_\ast(\mathcal{D}(X))$. This gives the inclusion $\mathcal{H}^{\tilde{f}_0}(Y)\subseteq P^f(Y,f(x_0))$ and the proof is complete.
\end{proof}

Let $H$ be a finite group acting freely on a $(2n+1)$-homotopy sphere $\Sigma(2n+1)$.
If $\Sigma(2n+1)/H$ is the corresponding space form then, following \cite[Chapter~VII, Proposition~10.2]{brown}, the action of $H=\mathcal{D}(\Sigma(2n+1)/H)$  on  $\pi_m(\Sigma(2n+1)/H,y_0)$ is trivial for $m>1$. In particular, $H$ acts trivially on $\pi_{2n+1}(\Sigma(2n+1)/H,y_0)\approx \pi_{2n+1}(\Sigma(2n+1),\tilde{y}_0)$. This implies that for any $h\in H$, the induced homeomorphism $h_*\colon \Sigma(2n+1)\to \Sigma(2n+1)$ is homotopic to $\mathrm{id}_{\Sigma(2n+1)}$. Consequently, if $f\colon X\to \Sigma(2n+1)/H$ is a map then $\mathcal{H}^{\tilde{f}_0}(\Sigma(2n+1)/H)=\mathcal{Z}_{H}f_*(\mathcal{D}(X))$. Because $P^f(\Sigma(2n+1)/H,f(x_0))\subseteq \mathcal{Z}_{H}f_*(\mathcal{D}(X))$, Proposition~\ref{pp} yields \[\mathcal{H}^{\tilde{f}_0}(\Sigma(2n+1)/H)=P^f(\Sigma(2n+1)/H,f(x_0))=\mathcal{Z}_{H}f_*(\mathcal{D}(X)). \]

Further, the result \cite[Theorem~1.17]{gol}, Theorem~\ref{qq} and Proposition~\ref{pp} lead to:

\begin{corollary}
If $f\colon X\to\Sigma(2n+1)/H $ is a map as in \cite[Theorem~1.14]{gol} then $\mathcal{D}^{\tilde{f}_0}(\Sigma(2n+1)/H)=\mathcal{H}^{\tilde{f}_0}(\Sigma(2n+1)/H)=\mathcal{Z}_Hf_*(\mathcal{D}(X))$. In particular, $J(f,x_0)=\mathcal{Z}_Hf_*(H)$ for any self-map $f\colon \Sigma(2n+1)/H\to \Sigma(2n+1)/H$.
\end{corollary}

Given a free action  of a finite group $H$ on $\mathbb{S}^{2n+1}$, Oprea \cite[\textsc{Theorem~A}]{oprea} has shown that $G(\mathbb{S}^{2n+1}/H,y_0)=\mathcal{Z}H$, the center of $H$.
In the special case of a free linear action of $H$ on $\mathbb{S}^{2n+1}$, the description of $G(\mathbb{S}^{2n+1}/H,y_0)$ via deck transformations presented in \cite[Theorem~II.1]{gottlieb1} has been applied  in \cite{bro} to get a very nice representation-theoretic proof of \cite[\textsc{Theorem~A}]{oprea}. The result stated in Theorem~\ref{qq} might be applied to extend the methods from \cite{bro} to simplify the  proof of  \cite[Theorem~1.17]{gol} on the case of a free linear action  of $H$ on $\mathbb{S}^{2n+1}$ as well.


\begin{thebibliography}{99}
\bibitem{bro}Broughton, S.A., \textit{The Gottlieb group of finite linear quotients of odd-dimensional spheres},
Proc.\ Amer.\ Math.\ Soc.\ \textbf{111}, no. 4 (1991), 1195--1197. 
\bibitem{brown} Brown, K.S., ``Cohomology of groups'', Springer, New York (1982).
\bibitem{james}Crabb, M.\ and James, I., ``Fibrewise homotopy theory''. Springer Monographs in Mathematics. Springer-Verlag London, 1998.
\bibitem{gol}Golasi\'nski, M.\ and de Melo, T., \textit{Generalized Gottlieb and Whitehead center groups of space forms} (submitted).
\bibitem{gottlieb1} Gottlieb, D., \textit{A certain subgroup of the fundamental group}, Amer.\ J.\ of Math., \textbf{87} (1965), 840--856.
\bibitem{gottlieb} Gottlieb, D., \textit{Evaluation subgroups of homotopy groups,} Amer.\ J.\ of Math.\ {\bf 91} (1969), 729--756.
\bibitem{jiang}Jiang, B-J., ``Lectures on Nielsen fixed point theory'', Contemp.\ Math., \textbf{14}, Amer.\ Math.\ Soc., Providence, R.I.\ (1983).
\bibitem{kim} Kim, R.\ and Oda, N., \textit{The set of cyclic-element preserving maps}, Topology Appl.\ \textbf{160}, no.\ 6 (2013), 794--805. 
\bibitem{oprea} Oprea, J., \textit{Finite group actions on spheres and the Gottlieb group}, J.\ Korean Math.\ Soc.\ \textbf{28}, no.\ 1 (1991), 65--78.
\end{thebibliography}
\end{document}